\documentclass[11pt,a4paper]{amsart}

 \usepackage[utf8x]{inputenc}
\usepackage{latexsym}
\usepackage{color,graphicx,shortvrb}
\usepackage{amsmath, amssymb}
\usepackage{amsfonts}
\usepackage[colorlinks, bookmarks=true]{hyperref}

\newtheorem{theorem}{Theorem}[section]
\newtheorem{lemma}[theorem]{Lemma}

\theoremstyle{definition}

\author{J. M. Almira}
\title{
An elementary inductive proof that $AB=I$ implies $BA=I$ for matrices
}

\begin{document}



\address{Departamento de Matem\'{a}ticas, Universidad de Ja\'{e}n, E.P.S. Linares,  Campus Cient\'{\i}fico Tecnol\'{o}gico de Linares, Cintur\'{o}n Sur s/n, 23700 Linares, Spain}
\email{jmalmira@ujaen.es}

\begin{abstract}
In this note we give an elementary demonstration of the fact that $AB=I_n$ implies $BA=I_n$ for square matrices $A,B$ with coefficients in a field $\mathbb{K}$. \end{abstract}
\maketitle

\markboth{J. M. Almira}{ $AB=I$ implies $BA=I$}
\section{Introduction}
Let $\mathbb{K}$ be a division ring and $M_n(\mathbb{K})$ be the ring of square matrices of order $n$ with coefficients in $\mathbb{K}$. Let us denote by $I_n$ de identity matrix of orden $n$, which is the unit element of $M_n(\mathbb{K})$. A very basic important fact about matrices is that they are non-commuting objects, even if $\mathbb{K}$ is commutative. This is important because matrices naturally appear in many applications where their non-commutativity is a key ingredient (I am thinking about Quantum Mechanics but there are thousands of other examples for this claim).  Thus, from the simple assumption that $AB=I_n$ for two square matrices $A,B\in M_n(\mathbb{K})$, it is not self evident that $BA=I_n$ since $AB$ and $BA$ may differ. Thus,  when introducing invertible linear maps or invertible matrices in a Linear Algebra introductory course, it is usual to force the equalities $AB=BA=I_n$ to tell that matrices $A,B$ are  invertible and $B=A^{-1}$ (see, e.g., \cite[Section 2.3]{N}, \cite[page 214]{A}, and \cite[page 25]{T}. In spite of the strong differences between all these books, the same path for a definition of invertible matrices is adopted). On the other hand, the following result  is well known and classic \cite[page 101]{vW} (see also  \cite[page 14]{HJ}):

\begin{theorem} \label{main} If $\mathbb{K}$ is a field and $A,B\in M_n(\mathbb{K})$, then $AB=I_n$ implies $BA=I_n$. 
\end{theorem}

Hence, for any field $\mathbb{K}$ and any natural number $n$,  we can claim that a square matrix $A\in M_n(\mathbb{K})$ is invertible (with inverse $B$) if and only if $AB=I_n$.

In this note we give an elementary demonstration of Theorem \ref{main}. By ``elementary" we mean that our proof follows from the very definitions of matrix and product of a matrix, with no extra help of more sophisticated results, as the use of dimensions of vector spaces or other ring- theoretical properties, as being Noetherian. The proof is also elementary in the sense that it relies on the concept and properties of the so called elementary operations on matrices. Finally, and no less important, this proof can be faced by any good student of a first year course in Mathematics, Physics or Engineering.

\section{The proof}
For our proof it is important to know how Gaussian row-reduction elimination process works. Concretely, it is important to observe that, for the elementary matrices associated to elementary operations a direct concept of inverse can be introduced. By this I mean that, for example, if $E_{i,j}$ denotes the matrix which results from interchanging the $i$-th and $j$-th rows of the identity matrix $I_n$, then it is clear that $E_{i,j}E_{i,j}=I_{n}$ so that we can define $E_{i,j}^{-1}=E_{i,j}$. Similar considerations can be done for the other types of elementary matrices: the elementary operation that reverses the work done by a given elementary operation is well known, is unique, and can be used to define the inverse matrix of the elementary matrix associated to the given elementary operation. Thus, for every elementary matrix $E$ we perfectly know what we mean by $E^{-1}$.  Moreover, we also need to use the following known fact: when we perform an elementary operation whose associated elementary matrix is $E$ to a given matrix $A$, the new matrix we get is $A^*=EA$.

\begin{lemma} \label{uno}
Let $A,B\in M_n(\mathbb{K})$ and assume $A=\left[\begin{array}{ccc} \alpha &   v^t \\ u & \widetilde{A}\end{array}\right]$ and $B=\left[\begin{array}{ccc} \beta &   w^t \\ h & \widetilde{B}\end{array}\right]$. Then 
\[
AB= \left[\begin{array}{ccc} \alpha \beta +v^th &   \alpha w^t+v^t\widetilde{B} \\ \beta u +\widetilde{A} h & uw^t+ \widetilde{A} \widetilde{B}\end{array}\right]
\]
 \end{lemma}
 
 \begin{proof}
 This is a direct computation. 
 \end{proof}

\begin{lemma} \label{dos} Assume that $A,B\in M_n(\mathbb{K})$ satisfy $AB=I_n$. If $Ax=0$, then $x=0$. 
\end{lemma}

\begin{proof}
We proceed by induction on $n$. The result trivially holds for $n=1$, since $AB=1$ implies $A\neq 0$ for scalars. Hence $Ax=0$ implies $x=0$. Assume $n>1$, decompose $A,B$ as the formulas given in Lemma \ref{uno} and impose $AB=I_n$. We decompose our proof in two cases:

\noindent \textbf{Case 1. } $A=\left[\begin{array}{ccc} 0 &   v^t \\ 0 & \widetilde{A}\end{array}\right]$.

In this case, we have that $\alpha=0$ and $u=0$, so that Lemma \ref{uno} tell us that 
\[
I_n=AB= \left[\begin{array}{ccc} v^th &   v^t\widetilde{B} \\  \widetilde{A} h &  \widetilde{A} \widetilde{B}\end{array}\right] =
 \left[\begin{array}{ccc} 1 &   0 \\  0 &  I_{n-1}\end{array}\right]
\]
It follows that $ \widetilde{A} \widetilde{B}=I_{n-1}$ and $ \widetilde{A} h=0$, $v^th=1$. Now, the induction hypothesis lead us to conclude that $h=0$, which contradicts $v^th=1$. Hence $AB=I_{n}$ can't hold if the first column of $A$ contains only zeros.

\noindent \textbf{Case 2. } The first column of $A$ is not identically zero. 

In this case we can perform elementary operations with associated elementary matrices $E_1,\cdots, E_t$  to transform the matrix $A$ into a matrix $A^*$ of the form 
$A^*= E_1\cdots E_t A= \left[\begin{array}{ccc} 1 &   (v^*)^t \\ 0 & \widetilde{A^*}\end{array}\right]$.  Then $B^*=BE_t^{-1}\cdots E_1^{-1}$ satisfies 
\begin{equation} \label{truco1}
A^*B^*=E_1\cdots E_t AB E_t^{-1}\cdots E_1^{-1}= E_1\cdots E_t \cdot E_t^{-1}\cdots E_1^{-1} =I_n.
\end{equation}
Let us write $B^*= \left[\begin{array}{ccc} \beta^* &   (w^*)^t \\ h^* & \widetilde{B^*}\end{array}\right]$. Then
\begin{equation}  \label{truco2}
\begin{array}{llll}
 \left[\begin{array}{ccc} 1 &   0 \\ 0 & I_{n-1} \end{array} \right]  &= &  I_n = A^*B^*=  \left[\begin{array}{ccc} 1 &   (v^*)^t \\ 0 & \widetilde{A^*}\end{array}\right]  \left[\begin{array}{ccc} \beta^* &   (w^*)^t \\ h^* & \widetilde{B^*}\end{array} \right] \\
&=& \left[\begin{array}{ccc}  \beta^* +(v^*)^th^* &    (w^*)^t+(v^*)^t\widetilde{B^*} \\  \widetilde{A^*} h^* &  \widetilde{A^*} \widetilde{B^*}\end{array}\right]
\end{array}
\end{equation}
In particular, we have that $\widetilde{A^*} \widetilde{B^*}=I_{n-1}$,  and $\widetilde{A^*} h^*=0$, so that the induction hypothesis implies that $h^*=0$ and the identity 
$ \beta^* +(v^*)^th^* =1$ is reduced to $ \beta^*=1$. 
Assume that $Ax=0$, with $x^t=[x_1,x^*]$. Then $A^*x=E_1\cdots E_t A x=E_1\cdots E_t 0=0$. Hence
\[
0=A^*x= A^*\left[\begin{array}{ccc} x_1 \\  x^* \end{array}\right] =   \left[\begin{array}{ccc} 1 &   (v^*)^t \\ 0 & \widetilde{A^*}\end{array}\right] \left[\begin{array}{ccc} x_1 \\  x^* \end{array}\right]  =  \left[\begin{array}{ccc} x_1 +   (v^*)^t x^* \\  \widetilde{A^*}x^*\end{array}\right] 
\]
It follows that $ \widetilde{A^*}x^*=0$ and the induction hypothesis implies that $x^*=0$. Hence $0=x_1+ (v^*)^t x^* =x_1$. This proves $x=0$.
\end{proof}
\begin{lemma} \label{tres}
If $AB=I_n$ then the first column of $A$ can't be the zero vector.
\end{lemma}
\begin{proof} If the first column of $A$ vanishes, then $Ae_1=0$, where $e_1$ denotes the vector $[1,0,\cdots,0]^t\in\mathbb{K}^n$, which is not the zero vector. This contradicts Lemma \ref{dos}. 
\end{proof}
Now we can demonstrate the main result of this paper:

\begin{proof}[Proof of Theorem \ref{main}]
It follows from Lemma \ref{tres} that the first column of $A$ does not vanish identically. Hence we can proceed as in the proof of Case 2 in Lemma \ref{dos} to construct the matrices $A^*= E_1\cdots E_t A= \left[\begin{array}{ccc} 1 &   (v^*)^t \\ 0 & \widetilde{A^*}\end{array}\right]$ and $B^* =BE_t^{-1}\cdots E_1^{-1}$. Obviously, these matrices satisfy  \eqref{truco1} and   \eqref{truco2}.  
Moreover, 
\[
B^*A^*=B E_t^{-1}\cdots E_1^{-1}E_1\cdots E_t A= B A.
\]
In particular,  \eqref{truco2} leads to $\widetilde{A^*} \widetilde{B^*}=I_{n-1}$, and the induction hypothesis implies that $\widetilde{B^*} \widetilde{A^*}=I_{n-1}$ too. Now Lemma \ref{dos} and the equality $\widetilde{A^*} h^*=0$ imply that $h^*=0$. Using this fact on the identity $\beta^* +(v^*)^th^*=1$ leads to $\beta^*=1$. Furthermore, we also have that 
\begin{equation}\label{truco}
(w^*)^t+(v^*)^t\widetilde{B^*} =0.
\end{equation}
Let us now consider the product $BA$:
\begin{eqnarray*}
BA &=& B^*A^*=   \left[\begin{array}{ccc} 1 &   (w^*)^t \\ 0 & \widetilde{B^*}\end{array} \right] \left[\begin{array}{ccc} 1 &   (v^*)^t \\ 0 & \widetilde{A^*}\end{array}\right]  \\
&=&  \left[\begin{array}{ccc} 1 &   (v^*)^t+ (w^*)^t  \widetilde{A^*}   \\ 0   &  \widetilde{B^*}\widetilde{A^*} \end{array}\right] \\
&=&  \left[\begin{array}{ccc} 1 &   (v^*)^t+ (w^*)^t  \widetilde{A^*}   \\ 0   &  I_{n-1} \end{array}\right] .\\
\end{eqnarray*}
It follows from \eqref{truco} that 
\begin{eqnarray*}
0 &=& ((w^*)^t+(v^*)^t\widetilde{B^*})\widetilde{A^*} \\
&=& (w^*)^t\widetilde{A^*}+(v^*)^t\widetilde{B^*}\widetilde{A^*} \\
&=& (w^*)^t\widetilde{A^*}+(v^*)^tI_{n-1}\\
&=& (v^*)^t+(w^*)^t\widetilde{A^*},
\end{eqnarray*}
which implies that 
\[
BA=B^*A^*=  \left[\begin{array}{ccc} 1 &   0 \\ 0 & I_{n-1} \end{array} \right]  =   I_n. 
\]
This ends the proof.
\end{proof}

\section{Other proofs}

\subsection{A proof based on dimension of vector spaces}

By fixing a basis of the vector space $\mathbb{K}^n$, we can identify in the natural way $M_n(\mathbb{K})$ with the vector space $End(\mathbb{K}^n)$.  Then the range of a matrix $A$ coincides with the dimension of the associated image space (which is spanned by columns of $A$). From $AB=I_n$ we conclude that $AB$ has range $n$ and, consequently, the range of $B$ must be also $n$ since, otherwise, the image space of $AB$
would be spanned by a set of at most $n-1$ vectors (just take a set of $n-1$ vectors spanning the range of $B$ and apply $A$ to these concrete vectors). Now,  $B=BI_n=B(AB)= (BA)B$ implies $(I_n-BA)B=0$ and, since $B$ has full rank, $I_n-BA$ is the null map. This means $BA=I_n$ and proves Theorem \ref{main}.

\subsection{A proof based on Noetherian property}
We take this proof from the well known abstract algebra website \cite{AA}.  A ring $R$ is called Dedekind-finite if for all $a,b \in R$, the identity $ab=1$ implies $ba=1$. 

\begin{theorem} Every (left or right) Noetherian ring R is Dedekind-finite.
\end{theorem}

\begin{proof} We will assume that $R$ is left Noetherian. Suppose that $ab=1$ for some $a,b \in R$. Define the map $f: R \longrightarrow R$  by $f(r)=rb$. Clearly $f$ is an $R$-module homomorphism and is onto because $f(ra)=(ra)b=r(ab)=r$, for all $r \in R$. Now we have an ascending chain of left ideals of $R$
\[
\ker f \subseteq \ker f^2 \subseteq \cdots.
\]
Since $R$ is left Noetherian, this chain stabilizes at some point, which means that  there exists some $n$ such that $\ker f^n = \ker f^{n+1}$. Clearly $f^n$ is onto because $f$ is onto. Thus $f^n(c)=ba-1$  for some $c \in R$. Then $f^{n+1}(c)=f(ba-1)=(ba-1)b=b(ab)-b=0$. Hence $c \in \ker f^{n+1}=\ker f^n$ and therefore $ba-1=f^n (c) = 0$. This ends the proof. \end{proof}

\begin{theorem} \label{algebra} Let $\mathbb{K}$ be a field and let $R$ be a finite dimensional $\mathbb{K}$-algebra. Then $R$ is Dedekind-finite.\end{theorem}

\begin{proof}  Every left ideal of $R$ is  a $\mathbb{K}$-vector subspace of $R$ and thus, since $\dim_{\mathbb{K}} R < \infty$, any ascending chain of left ideals of $R$ will stop at some point. So $R$ is left Noetherian and thus, it is also Dedekind-finite. \end{proof}

Now Theorem \ref{main} follows as a corollary of Theorem \ref{algebra} since $M_n(\mathbb{K})$ is a  finite dimensional $\mathbb{K}$-algebra. What is more, we can conclude that  $M_n(R)$ is Dedekind-finite for any commutative domain $R$  because $M_n(R)$ is a subring of $M_n(\hat{R})$, where $\hat{R}$ is the quotient field of $R$.

\section{The result does not hold for operators on infinite dimensional vector spaces}

Let $V$ be any infinite dimensional vector space over the field $\mathbb{K}$. Let $\{e_k\}_{k=1}^{\infty}$ be a linear independent set in $V$ and consider $\beta= \{e_k\}_{k=1}^{\infty}\cup \{w_j\}_{j\in J}$ a basis of $V$ which contains $e_k$ for all $k$. We introduce the linear maps $A,B\in End(V)$ defined by 
\[
\left \{\begin{array}{ccc} A(e_k)=e_{k+1} & \text{for all }  k=1,2,\cdots \\
A(w_j)=w_j & \text{ for all } j\in J 
\end{array} \right . 
\]
and 
\[
\left \{\begin{array}{ccc} 
 B(e_1)=0 &  \\
B(e_k)=e_{k-1} & \text{for all }  k=2,3, \cdots \\
B(w_j)=w_{j} & \text{ for all } j\in J 
\end{array} \right . 
\]

Then $BA=1_d$ and $ABe_1=A0=0$, so that $AB\neq 1_d$.  This, jointly with Theorem \ref{main}, proves the following
\begin{theorem} Let $V$ be a vector space. Then $End(V)$ is Dedekind-finite iff $\dim V<\infty$.
\end{theorem}

\bibliographystyle{amsplain}


\end{document}